\documentclass[12pt,reqno]{amsart}
\usepackage{amssymb, a4wide, amsmath, mathtools,xy}
\xyoption{all}
\usepackage{latexsym,pdfsync,xcolor,graphicx}
\usepackage{multirow}

\usepackage[T1]{fontenc}

\usepackage[utf8]{inputenc}

\usepackage{array}
\usepackage{upgreek}
\usepackage{pstricks-add}
\usepackage{enumerate}
\usepackage{orcidlink}

\usepackage[english]{babel}	
\usepackage{comment}
\allowdisplaybreaks

\usepackage{pgf,tikz}
\usepackage{wrapfig}
\usetikzlibrary{arrows}

\usepackage{float}
\usepackage{appendix}

\usepackage{hyperref}
\usepackage{rotating}

\theoremstyle{plain}%
\newtheorem{theorem}{Theorem}[section]
\newtheorem{proposition}[theorem]{Proposition}%
\newtheorem{lemma}[theorem]{Lemma}%
\newtheorem{corollary}[theorem]{Corollary}%

\theoremstyle{definition}
\newtheorem{definition}[theorem]{Definition}%
\newtheorem{example}[theorem]{Example}%

\theoremstyle{remark}
\newtheorem{remark}[theorem]{Remark}

\DeclareMathOperator{\ann}{ann}
\DeclareMathOperator{\spa}{\textup{\textsf{span}}}

\DeclareMathOperator{\rank}{rank}
\DeclareMathOperator{\supp}{supp}

\DeclareMathOperator{\codim}{codim}

\providecommand{\abs}[1]{\lvert#1\rvert}

\numberwithin{equation}{section}
\title[On the subalgebra lattice of solvable evolution algebras]{On the subalgebra lattice of solvable evolution algebras}

\author[M. Ladra]{Manuel Ladra\textsuperscript{1}\,\orcidlink{0000-0002-0543-4508}}
\address{\textsuperscript{1}Department of Mathematics \& CITMAga, Universidade de Santiago de Compostela, 15782 Santiago de Compostela, Spain}
\email{manuel.ladra@usc.es}

\author[P. Páez-Guillán]{Pilar Páez-Guillán\textsuperscript{2}\,\orcidlink{0000-0003-2761-7505}}
\address{\textsuperscript{2}Department of Mathematics, Universidade de Santiago de Compostela, 15782 Santiago de Compostela, Spain}
\email{ppaez51@gmail.com}

\author[A. Pérez-Rodríguez]{Andrés Pérez-Rodríguez\textsuperscript{3}\,\orcidlink{0009-0007-1095-5328}}
\address{\textsuperscript{3}Department of Mathematics \& CITMAga, Universidade de Santiago de Compostela, 15782 Santiago de Compostela, Spain}
\email{andresperez.rodriguez@usc.es}



\begin{document}
	
	\begin{abstract}
		The main objective of this paper is to study the relationship between a solvable evolution algebra and its subalgebra lattice, emphasizing two of its main properties: distributivity and modularity. First, we will focus on the nilpotent case, where distributivity is characterised, and a necessary condition for modularity is deduced. Subsequently, we comment on some results for solvable non-nilpotent evolution algebras, finding that the ones with maximum index of solvability have the best properties. Finally, we characterise modularity in this particular case by introducing supersolvable evolution algebras and computing the terms of the derived series.  
	\end{abstract}

\maketitle

\section{Introduction}
Lattice theories have been developed in several algebraic structures by mainly working with two lattice-theoretic classical properties defined by identities: \textit{distributivity} and \textit{modularity}. By a well-known theorem of Birkhoff (\cite[Theorem 2.1.1]{Gr_98}), a lattice is distributive if and only if it does not contain a diamond or a pentagon as a sublattice. 
Similarly, Dedekind characterised a modular lattice if and only if it does not contain a pentagon as a sublattice (see \cite[Theorem 2.1.2]{Gr_98}). 
Particularly, the relationship between the properties of a group and the structure of its subgroup lattice has been deeply studied and found interesting in group theory (see \cite{Sc_94}). Ore's Theorem, which states that the subgroup lattice of a group $G$ is distributive if and only if $G$ is locally cyclic or the fact that the lattice of normal subgroups (and consequently the subgroup lattice of an abelian group) is modular highlight such a strong connection and make the case of groups especially representative. 

Along the same line, the lattice of right ideals of a ring (see \cite{Br_76}) and the lattice of submodules of a module (see \cite{Ca_75}) have also been investigated. Moreover, the lattice of ideals of an associative algebra has been addressed in \cite{Ja_57} due to its importance in representation theory. Last but not least, this kind of studies has also been carried out in non-associative structures such as Lie (see \cite{Ge_76,Ko_65}) or Leibniz (see \cite{ST_22}) algebras, and even in the context of restricted Lie algebras (see \cite{MPS_21,PST_23}). However, this relationship is not well known in genetic algebras (see \cite{MS_00}) and is completely unsettled in the particular case of evolution algebras. 

Evolution algebras are commutative non-associative structures which appeared for the first time in $2006$ (see \cite{TV_06}). In this paper, J. P. Tian and P.  Vojt\v{e}chovsk\'{y} showed how the multiplication of an evolution algebra mimics the self-reproduction rule of non-Mendelian genetics.
Two years later, J. P. Tian (\cite{Tian_08}) published a new monograph in which the properties and biological applications of the evolution algebras are studied in more detail.
Despite being an emerging research topic, numerous connections between evolution algebras and other fields of mathematics have been explored. For instance, \textit{graph theory} turns out to be a useful tool when studying the properties of evolution algebras (see \cite{EL_15,EL_21}). In addition, the relationship between evolution algebras and \textit{Markov chains} has been studied (see \cite{Pa_21}) and, in the context of \textit{group theory}, the question of whether any finite group may be realised as the automorphism group of an evolution algebra has also been developed (see \cite{CLTV_21,CMTV_23}).

Hence, the motivation for this paper is to establish such relations between evolution algebras and \textit{lattice theory}, particularly the investigation of how far the properties of an evolution algebra determine those of its subalgebra lattice and vice versa. Mainly, we centre our attention on solvable evolution algebras, in which distributivity and modularity will be studied by using, inter alia, the notion of quasi-ideal, supersolvability and the terms of the derived series.

The text is structured into six sections. After this introductory one, Section~\ref{preliminaries} is devoted to preliminaries. We recall some concepts on general evolution algebras and then we focus on the solvable case by reviewing nilpotent evolution algebras and introducing two wide families that contain solvable non-nilpotent evolution algebras. Finally, we revise the necessary lattice-theoretical framework for the development of the paper.

Sections~\ref{maxindexsolv} and \ref{supersoluble} pioneer the treatment of solvable evolution algebras with maximum index of solvability and supersolvable evolution algebras, respectively. In the first one, we characterise solvable evolution algebras with maximum index of solvability in terms of their derived subalgebras. Furthermore, we present several results about their subalgebras, which will be instrumental when studying their subalgebra lattice. In the second one, we show that supersolvability is a sufficient condition for lower semimodularity, and we characterise it in solvable evolution algebras.

In Section~\ref{nilpotentcase}, we focus on the nilpotent case. As the structure matrix of a nilpotent evolution algebra is strictly triangular, we characterise distributivity in terms of its index of nilpotency. Moreover, we deduce a necessary condition for modularity, which turns out to be a characterisation under an easily checkable additional hypothesis related with the annihilator of the algebra.

Finally, Section~\ref{solvablecase} is devoted to studying the subalgebra lattice of solvable non-nilpotent evolution algebras. We finish the paper by characterising the modularity of solvable evolution algebras with maximum index of solvability by using supersolvability.

\section{Background on evolution algebras and lattice theory}\label{preliminaries}
First, we establish some basic notations that will be useful throughout this paper. $\mathbb{K}$ will denote an arbitrary field, and $\mathbb{K}^*$ will stand for $\mathbb{K}\backslash\{0\}$. Let $\mathcal{A}$ be a $\mathbb{K}$-algebra and $S\subseteq\mathcal{A}$ a subset. We will use $\spa\{S\}$ to denote the $\mathbb{K}$-linear span of $S$ and the symbol $+$ for sums of vector-space structures. At the same time, $\langle S\rangle$ will stand for the subalgebra of $\mathcal{A}$ generated by $S$. Particularly, if a subspace is closed for the product, then it will be a subalgebra.

\subsection{Generalities about evolution algebras}
An \textit{evolution algebra} over a field $\mathbb{K}$ is a $\mathbb{K}$-algebra $\mathcal{E}$ which admits a basis $B=\{e_1,\dots,e_n, \dots\}$ such that $e_ie_j=0$ for any $i\neq j$. Such a basis is called a \textit{natural basis}.  Notice that all the evolution algebras considered in this work will be finite dimensional, i.e. $B$ will be a finite set. Fixed a natural basis $B$ in $\mathcal{E}$, the scalars $a_{ik}\in\mathbb{K}$ such that $e_i^2=\sum_{k=1}^na_{ik}e_k$ will be called \textit{structure constants of $\mathcal{E}$ relative to $B$}, and the matrix $M_B(\mathcal{E})=(a_{ik})$ is said to be the \textit{structure matrix of $\mathcal{E}$ relative to $B$}.

The annihilator of an evolution algebra $\mathcal{E}$ is characterised by \cite[Proposition 1.5.3]{Ca_16}, $\ann(\mathcal{E})=\spa\{e_i\in B\colon e_i^2=0\}$. If $\ann(\mathcal{E})=0$, we say that $\mathcal{E}$ is \textit{non-degenerate}. Otherwise, we say that $\mathcal{E}$ is \textit{degenerate}. Given an element $u=\sum_{i=1}^{n}\mu_ie_i\in\mathcal{E}$ we define its \textit{support relative to $B$} as $\supp_B(u)\coloneqq\{i\in\{1,\dots,n\}\colon \mu_i\neq0\}$. Moreover, an element $u\in\mathcal{E}$ is called \textit{idempotent} if $u^2=u$ and \textit{absolute nilpotent} if $u^2=0$.  Finally, recall that an ideal $I$ of $\mathcal{E}$ is called a \textit{basic ideal of $\mathcal{E}$ relative to $B$} if it admits a natural basis consisting of vectors from $B$.

Knowing that evolution algebras are closed with respect to quotients by ideals (see~\cite[Lemma 1.4.11]{Ca_16}), the following remark will be useful in the end of this paper.

\begin{remark}\label{rem:quo_bas_i}
	Let $\mathcal{U}$ be a subalgebra of an evolution algebra $\mathcal{E}$ and $I$ a basic ideal of $\mathcal{E}$ such that $I\subset U$. Then, it is easy to check that $U/I$ is a basic ideal of $\mathcal{E}/I$ if and only if $U$ is a basic ideal of $\mathcal{E}$.
\end{remark}

We focused our investigation on the nilpotent and solvable cases. Given an evolution algebra $\mathcal{E}$, we define the following sequences of subalgebras:
\begin{align*}
	\mathcal{E}^1&=\mathcal{E},\qquad  \quad  \ \ \mathcal{E}^{k+1}=\sum_{i=1}^k\mathcal{E}^i\mathcal{E}^{k+1-i};	\\
	\mathcal{E}^{(1)}&=\mathcal{E},\qquad \quad  \mathcal{E}^{(k+1)}=\mathcal{\mathcal{E}}^{(k)}\mathcal{E}^{(k)}.
\end{align*}
An evolution algebra $\mathcal{E}$ is called \textit{nilpotent} if there exists $n\in\mathbb{N}$ such that $\mathcal{E}^n=0$ (and the minimal such number is called \textit{index of nilpotency}), and \textit{solvable} if there exists $n\in\mathbb{N}$ such that $\mathcal{E}^{(n)}=0$ (and the minimal such number is called \textit{index of solvability}). As in Lie algebras, the sequence $0=\mathcal{E}^{(n)}\subsetneq\mathcal{E}^{(n-1)}\subsetneq\dots\subsetneq\mathcal{E}^{(2)}\subsetneq\mathcal{E}$ will be called the \textit{derived series of $\mathcal{E}$} and, particularly, $\mathcal{E}^{(2)}$ will be called the \textit{derived subalgebra of $\mathcal{E}$}. Moreover, as introduced in \cite{FGN_22}, we define the \textit{principal powers} of an element $u\in\mathcal{E}$, recursively, as $u^1=u$, $u^n=u^{n-1}u$; and we define its \textit{plenary powers} as $u^{(0)}=u$, $u^{(n)}=u^{(n-1)}u^{(n-1)}$.

Recall that every nilpotent evolution algebra is also solvable but the converse is not true in general (\cite[Example 2.8]{CGOT_13}). A key feature of nilpotent evolution algebras is that their structure matrix can be assumed to be strictly (upper or lower) triangular by~\cite[Theorem~2.7]{CLOR14}.
Unlike the nilpotent case, there is no characterisation for solvable evolution algebras in terms of their structure matrix. Consequently, we will define two families of solvable evolution algebras that will be instrumental in the study of their subalgebra lattice. Notice that both families contain non-nilpotent evolution algebras.

The first family has been originally introduced in \cite{CGOT_13}.
Let $\mathbb{K}$ be a field of characteristic  different from $2$ and $V$ a vector space over $\mathbb{K}$ with basis $\{e_1,\dots,e_n\}$. Given scalars $\lambda_1,\dots,\lambda_n\in\mathbb{K}$ such that $\sum_{j=1}^{n}\abs{\lambda_j}\neq0$ and $\sum_{j=1}^{k}\lambda_j=0$ for a $k\in\{1,\dots,n\}$, we define the evolution algebra 
\begin{align}\label{def:solv_1}
	\mathcal{E}_k(\lambda_1,\dots,\lambda_n)\colon e_i^2=\lambda_i(e_1+\dots+e_k),\text{ for any } i\in\{1,\dots,n\}.
\end{align}

As explained in \cite[Remark 2.6]{CGOT_13}, the evolution algebras described in \eqref{def:solv_1} can be divided in two disjoint classes. First, when $\lambda_i=0$ for all $1\leq i\leq k$, this evolution algebra is nilpotent. The second one is when $\lambda_i\neq0$ for some $1\leq i\leq k$. Then, by a natural basis transformation, one can assume that $e_1^2=e_1+\dots+e_k$ and hence, these evolution algebras are not nilpotent.

The importance of this type of solvable algebras lies in the fact that every solvable algebra whose derived subalgebra is one-dimensional is isomorphic to one of them \cite[Proposition 2.5]{CGOT_13}.
Let $\mathcal{E}$ be an $n$-dimensional evolution algebra such that $\dim{\mathcal{E}^{(2)}}=1$. Then $\mathcal{E}^{(3)}=0$ if and only if $\mathcal{E}\cong\mathcal{E}_k(\lambda_1,\dots,\lambda_n)$ for some scalars $\lambda_1,\dots,\lambda_n\in\mathbb{K}$ and $k\in\mathbb{N}$.

Next, we introduce a new family of solvable evolution algebras which generalises~\eqref{def:solv_1}.
\begin{definition}\label{def:solv_2}
	Let $\mathcal{E}_1,\dots,\mathcal{E}_r$ be evolution algebras of the type described in \eqref{def:solv_1} with natural bases $B_1,\dots,B_r$, respectively. Suppose that $r\geq1$. Then, it will be said that an evolution algebra $\mathcal{E}$ belongs to the family $\mathcal{F}(\mathcal{E}_1,\dots,\mathcal{E}_r)$ if there exists a natural basis $B$ such that $\cup_{i=1}^rB_i\subseteq B$ and
	\begin{align}\label{matriz_solv}
		M_{B}(\mathcal{E})=\begin{pmatrix*}[l] A_{m\times m}&0_{m\times(n-m)}\\C_{(n-m)\times m}&L_{(n-m)\times(n-m)} \end{pmatrix*},
	\end{align}
	where $m=\abs{\cup_{i=1}^rB_i}$, $A$ is block diagonal (in which each block is $M_{B_i}(\mathcal{E}_i)$ for $1\leq i\leq r$), $L$ is strictly lower triangular and $C$ can be any matrix. Notice that $\oplus_{i=1}^r\mathcal{E}_i\subseteq\mathcal{E}$ is an ideal, and the quotient $\mathcal{E}/(\oplus_{i=1}^r\mathcal{E}_i)$ is a nilpotent evolution algebra.
\end{definition}
The following lemma shows that the solvable evolution algebras $\mathcal{E}_1,\dots,\mathcal{E}_r$ can be assumed to be non-nilpotent without loss of generality.
\begin{lemma}\label{lem:spg_mod}
	Let $\mathcal{E}_1,\dots,\mathcal{E}_r$ be evolution algebras of the type described in Definition~\ref{def:solv_1}. If for some $i\in\{1,\dots,r\}$, $\mathcal{E}_i$  is nilpotent, then $\mathcal{F}(\mathcal{E}_1,\dots,\mathcal{E}_r)\subset\mathcal{F}(\mathcal{E}_1,\dots,\widehat{\mathcal{E}_i},\dots,\mathcal{E}_r)$.
\end{lemma}
Finally, it is worth noting  that Definition~\ref{def:solv_2} describes a quite varied family of solvable evolution algebras since we can find solvable evolution algebras of any dimension and with any index of solvability. For instance, let $\mathcal{E}$ be an evolution algebra with natural basis $\{e_1,\dots,e_n\}$ and the following multiplication:
\[e_1^2=-e_2^2=e_1+e_2\qquad\text{and}\qquad e_i^2=e_{i-1}, \ \text{ for all } \ 3\leq i\leq k\leq n.\] 
$\mathcal{E}$ is a solvable non-nilpotent evolution algebra with index of solvability $k+1$.

\subsection{Basic ideas on lattice theory}
For the reader's convenience, we now recall  some basic lattice-theoretical definitions that will appear in our study, already in the framework of evolution algebras. Let $\mathcal{E}$ be an evolution algebra. We say that $\mathcal{E}$ is \textit{distributive} if, for all subalgebras $U,V,W$ of $\mathcal{E}$, one has $\langle U,V\cap W\rangle=\langle U, V\rangle\cap\langle U, W\rangle$, and that $\mathcal{E}$ is \textit{modular} if $\langle U,V\cap W\rangle=\langle U, V\rangle\cap W$ for all subalgebras $U,V,W$ of $\mathcal{E}$ with $U\subseteq W$. Notice that every distributive lattice is also modular, but the converse is not necessarily true. 
We say that a subalgebra $U$ of $\mathcal{E}$ is \textit{upper semimodular} in $\mathcal{E}$ if $U$ is maximal in $\langle U, V\rangle$ for every subalgebra $V$ of $\mathcal{E}$ such that $U\cap V$ is maximal in $V$, and that $U$ is \textit{lower semimodular} in $\mathcal{E}$ if $U\cap V$ is maximal in $U$ for every subalgebra $V$ of $\mathcal{E}$ such that $V$ is maximal in $\langle U, V\rangle$. The algebra $\mathcal{E}$ is called \textit{upper semimodular} (resp. \textit{lower semimodular}) if all of its subalgebras are upper semimodular (resp. lower semimodular) in $\mathcal{E}$. Notice that if $\mathcal{E}$ is modular, then it is upper and lower semimodular.

If $U,V$ are subalgebras of $\mathcal{E}$ with $U\subseteq V$, a \textit{$J$-series} (or \textit{Jordan-Dedekind series}) for $(U,V)$ is a series $U=U_0\subsetneq U_1\subsetneq\dots\subsetneq U_r=V$ of subalgebras such that $U_i$ is a maximal subalgebra of $U_{i+1}$ for $0\leq i\leq r-1$. This series has a \textit{length} equal to $r$. We shall call $\mathcal{E}$ a \textit{$J$-algebra} if all $J$-series for $(U,V)$ have the same finite length whenever $U$ and $V$ are subalgebras of $\mathcal{E}$ with $U\subsetneq V$.

With the purpose of studying modularity in evolution algebras,
we finish this section by stating the following results (see \cite[Theorem 2.2]{An_94} and \cite[Theorem 5.1]{ST_22}) in the wider context of solvable algebras. They follow from the existence of a complete flag of subalgebras in every finite-dimensional solvable algebra. Just recall that a subalgebra $U$ of $\mathcal{E}$ is called a \textit{quasi-ideal of $\mathcal{E}$} if $\langle U,V\rangle=U+V$ for all subalgebras $V$ of $\mathcal{E}$. Notice that every ideal is a quasi-ideal but the converse is not necessarily true (see below Example~\ref{ex:CR_algebra}, for instance).
\begin{proposition}[\cite{An_94}]\label{prop:equiv_mod_qi}
	Let $\mathcal{E}$ be a finite-dimensional solvable (non-necessarily evolution) algebra.
	Then, $\mathcal{E}$ is modular if and only if it is upper semimodular and if and only if every subalgebra is a quasi-ideal.
\end{proposition}
\begin{proposition}[\cite{ST_22}]\label{prop:equiv_ls}
	Let $\mathcal{E}$  be a finite-dimensional solvable (non-necessarily evolution) algebra. Then, $\mathcal{E}$ is lower semimodular if and only if it is a $J$-algebra.
\end{proposition}

\section{Evolution algebras with maximum index of solvability}\label{maxindexsolv}

In this section, we prove some results about solvable evolution algebras with maximum index of solvability, which will be helpful for our main results. Recall that if a solvable evolution algebra $\mathcal{E}$ of dimension $n$ has maximum index of solvability, then its index of solvability is $n+1$ (see \cite[Proposition~4.2]{CGOT_13}). First, we characterise these solvable evolution algebras.

\begin{proposition}\label{prop:caract_mis}
	Let $\mathcal{E}$ be an $n$-dimensional solvable evolution algebra. Then, $\mathcal{E}$ has maximum index of solvability if and only if $\codim{\mathcal{E}^2}=1$.
\end{proposition}
\begin{proof}
	If $\mathcal{E}$ has maximum index of solvability, then the derived series has $n$ non-zero terms, that is,
	\begin{equation}\label{dim}
		n=\dim{\mathcal{E}}>\dim{\mathcal{E}^{(2)}}>\dots>\dim{\mathcal{E}^{(n)}}>\dim{\mathcal{E}^{(n+1)}}=0.
	\end{equation}
	From~\eqref{dim}, it holds that $\dim{\mathcal{E}^{(2)}}=\dim{\mathcal{E}^2}=\dim{\big(\spa\{e_1^2,\dots,e_n^2\}\big)}=n-1$. Conversely, if $\codim\mathcal{E}^2=1$, then there exists a reordering of the natural basis such that $\{e_1^2,\dots,e_{n-1}^2\}$ is a linearly independent subset. Consider $\mathcal{E}^{(s)}=\spa\{u_1,\dots,u_l\}$, where $u_i=\sum_{j=1}^n\mu_{ij}e_j$ with $\mu_{ij}\in\mathbb{K}$. Without loss of generality, suppose that the matrix $(\mu_{ij})_{i,j=1}^{l,n}$ is in row echelon form. Moreover, at least one element exists such that
	$n$ belongs to its support, say $u_l$. Then, make zeros in the last column, performing elementary operations. Finally, after all these changes and considering that $\{e_1^2,\dots,e_{n-1}^2\}$ is a linearly independent subset, we have that the squares $u_i^2=\sum_{k=1}^{n-1}\mu_{ik}^2e_k^2$ with $i=1,\dots,l-1$ also constitute a linearly independent subset. As $s$ was taken arbitrarily, it holds that $\dim{\mathcal{E}^{(s+1)}}=\dim{\mathcal{E}^{(s)}}-1$ for any $1\leq s\leq n$, which yields the claim.
\end{proof}

The following result determines subalgebras of solvable evolution algebras with maximum index of solvability.

\begin{proposition}\label{prop:cant_sub}
	Let $\mathcal{E}$ be a solvable evolution algebra with maximum index of solvability over any field of characteristic different from $2$. Then, there exists a reordering of the natural basis such that $e_n^2=-\lambda_1^2e_1^2-\dots-\lambda_{m}^2e_{m}^2$ with $m\leq n-1$ and $\lambda_1,\dots,\lambda_m\in\mathbb{K}^*$, and $\spa\{\pm\lambda_1e_1\pm\dots\pm\lambda_{m}e_{m}\pm e_n\}$ (with the signs taken interchangeably) are all the one-dimensional subalgebras. Consequently, $\mathcal{E}$ has $2^m$ one-dimensional subalgebras.
\end{proposition}
\begin{proof}
	The first part is obvious. For the second part, we use induction on $m$. When $m=1$, it is trivially true because $\spa\{\lambda_1e_1+e_n\}$
	and $\spa\{\lambda_1e_1-e_n\}$ are the only one-dimensional subalgebras. Then suppose the assertion is true for $m$,
	i.e. there exist elements $u_i$ such that $u_i^2=0= e_n^2+\lambda_1^2e_1^2+\dots+\lambda_{m}^2e_{m}^2$ for $i=1,\dots,2^m$.
	Then, if $e_n^2=-\lambda_1^2e_1^2-\dots-\lambda_{m+1}^2e_{m+1}^2$,
	the elements $u_i+\lambda_{m+1}e_{m+1}$ and $u_i-\lambda_{m+1}e_{m+1}$ for $i=1,\dots,2^m$, span $2\cdot2^m=2^{m+1}$ different one-dimensional subalgebras.
	Finally, the fact that $\lambda_1,\dots,\lambda_m\in\mathbb{K}^*$ is a consequence of the fact that $\spa\{\pm\lambda_1e_1\pm\dots\pm\lambda_{m}e_{m}\pm e_n\}$ are all the one-dimensional subalgebras. Otherwise, there would not exist any one-dimensional subalgebra, which contradicts the existence of a complete flag made up of subalgebras in every solvable algebra.
\end{proof}
\begin{remark}\label{rem:mis}
	As $\lambda_1,\dots,\lambda_m\in\mathbb{K}^*$ in the previous lemma, with the suitable natural basis change ($f_i=\lambda_i e_i$ for any $i=1,\dots,m$) one can assume that $e_n^2=-e_1^2-\dots-e_m^2$ for some $m\leq n-1$. Particularly, the one-dimensional subalgebras are $\spa\{\pm e_1\pm\dots\pm e_{m}\pm e_n\}$.
\end{remark}

Notice that \cite[Proposition 4.2]{CGOT_13} establishes an equivalence between maximum index of nilpotency and maximum index of solvability in the nilpotent case. Moreover, \cite[Theorem~4.5]{CGOT_13} states that a nilpotent evolution algebra reaches its maximum index of nilpotency if and only if the first (upper or lower) diagonal of its structure matrix has no zero elements.

\begin{corollary}\label{cor:sub_chain}
	Let $\mathcal{E}$ be a nilpotent evolution algebra with maximum index of nilpotency, $B=\{e_1,\dots,e_n\}$ a natural basis of $\mathcal{E}$ such that $M_B(\mathcal{E})$ is strictly upper triangular and $U$ a subalgebra. Then $U=\spa\{e_k,\dots,e_n\}$ for some $k=1,\dots,n$.
\end{corollary}
\begin{proof}
	First, by Proposition~\ref{prop:cant_sub}, since $e_n^2=0$, the only one-dimensional subalgebra is $\spa\{e_n\}$. 
	To study the other subalgebras,  consider an arbitrary element $u=\sum_{i=1}^{n}\mu_ie_i\in\mathcal{E}$ and define $k=\min\{\supp(u)\}$. Hence, we have that
	\begin{align*}
		u^{(0)}&=u=\mu_ke_k+v_{k+1},\ v_{k+1}\in\spa\{e_{k+1},\dots,e_n\};\\
		&\:\ \vdots \\
		u^{(i)}&=\mu_k^{2^{i}}a_{k(k+1)}^{2^{i-1}}\dots a_{(k+i-1)(k+i)}e_{k+i}+v_{k+i+1}, \ v_{k+i+1}\in\spa\{e_{k+i+1},\dots,e_n\};\\
		&\:\ \vdots \\
		u^{(n-k)}&=\mu_k^{2^{n-k}}a_{k(k+1)}^{2^{n-k-1}}\dots a_{(n-1)n}e_{n};
	\end{align*}
	with $\mu_k^{2^{i}}a_{k(k+1)}^{2^{i-1}}\dots a_{(k+i-1)(k+i)}\neq0$ for all $0< i\leq n-k$. As all of the previous elements are in $\langle u\rangle$, it holds that $e_n\in\langle u\rangle$. In the same way, as $u^{(n-k-1)}=c_1e_{n-1}+c_2e_n\in \langle u\rangle$ with $c_1,c_2\in\mathbb{K}$, $c_1\neq0$, and $e_n\in\langle u\rangle$, it also holds that $e_{n-1}\in\langle u\rangle$. Inductively, it comes that $\langle u\rangle=\spa\{e_k,\dots,e_n\}$.
\end{proof}

Finally, the following two lemmas will also be instrumental throughout our study.

\begin{lemma}\label{lem:sub_2}
	Let $\mathcal{E}$ be an $n$-dimensional solvable evolution algebra with maximum index of solvability. Then $\mathcal{E}^{(n-1)}$ has at most two subalgebras.
\end{lemma}
\begin{proof}
	By~\eqref{dim}, we can assume that $\mathcal{E}^{(n)}=\spa\{v\}$ and $\mathcal{E}^{(n-1)}=\spa\{v,u\}$, where  $v^2=0$, $u^2=k_1v$ and $vu=uv=k_2v$ with $k_1,k_2\in\mathbb{K}$. Hence, any subalgebra of $\mathcal{E}^{(n-1)}$ will be spanned by $\alpha v+\beta u$ such that 
	\[0=(\alpha v+\beta u)^2=\alpha^2 v^2+\beta^2 u^2+ 2\alpha\beta vu=\beta(\beta k_1+2\alpha k_2)v.\]
	Then, it is easy to deduce that, if $k_1\neq0$, $\mathcal{E}^{(n-1)}$ has at most two different subalgebras: $\spa\{v\}$ and $\spa\{v-\frac{2k_2}{k_1}u\}$; and if $k_1=0$ it has two different subalgebras: $\spa\{v\}$ and $\spa\{u\}$.
\end{proof}
\begin{lemma}\label{lem:clas_solv_anyfield}
	Any two-dimensional solvable evolution algebra over any field $\mathbb{K}$ is abelian or isomorphic to one of the following non-isomorphic algebras:
	\begin{enumerate}[\rm (i)]
		\item $e_1^2=-e_2^2=e_1+e_2$;
		\item $e_1^2=e_2$, $e_2^2=0$ 
	\end{enumerate}
\end{lemma}
\begin{proof}
	Notice that if $\mathcal{E}$ is a two-dimensional solvable evolution algebra, then its product is given by $e_1^2=\alpha e_1+\beta e_2$ and $e_2^2=-k(\alpha e_1+\beta e_2)$, with $\alpha,\beta,k\in\mathbb{K}$. We have the following possibilities:
	\begin{enumerate}[\rm (a)]
		\item If $\alpha\neq0$ and $\beta\neq0$, then we can take an appropriate change of basis such that $f_1^2=-f_2^2=f_1+f_2$.
		\item If $\alpha=0$ and $\beta\neq0$, then $k=0$ and again, we can take an appropriate change of basis such that $f_1^2=f_2$ and $f_2^2=0$.
		\item If $\alpha=0$ and $\beta=0$ then $\mathcal{E}$ is abelian.
	\end{enumerate}
\end{proof}
\section{Supersolvable evolution algebras}\label{supersoluble}

When studying modularity, we will naturally find another type of evolution algebras. An evolution algebra $\mathcal{E}$ is said to be \textit{supersolvable} if there exists a complete flag made up of ideals, that is, there exists a chain $0=I_0\subsetneq I_1\subsetneq\dots\subsetneq I_n=\mathcal{E}$ of ideals such that $\dim{I_i}=i$ for every $0\leq i\leq n$. For example, notice that every nilpotent evolution algebra is supersolvable. If we consider, without loss of generality, that its structure matrix is strictly upper triangular, then
\[
0\subsetneq\spa\{e_n\}\subsetneq\spa\{e_{n-1,}e_n\}\subsetneq\dots\subsetneq\spa\{e_2,\dots,e_n\}\subsetneq\mathcal{E}
\]
is trivially a complete flag made up of ideals. In fact, every solvable evolution algebra with a structure matrix \eqref{matriz_solv} is supersolvable.
\begin{proposition}
	Let $\mathcal{E}$ be a solvable evolution algebra of the kind described in Definition~\ref{def:solv_2}. Then, $\mathcal{E}$ is supersolvable.
\end{proposition}
\begin{proof}
	Let $\mathcal{E}$ be a solvable evolution algebra, as stated, and we will use induction on $\dim{\mathcal{E}}$. As $r\geq1$, we can consider the one-dimensional ideal $\mathcal{E}_1^2$. In fact, it is easy to check that $\mathcal{E}/\mathcal{E}_1^2$ is nilpotent or, again, solvable of such kind. The result follows by the induction hypothesis.
\end{proof}
\begin{remark}
	There exist supersolvable evolution algebras which are not solvable. Just consider the regular evolution algebra $\mathcal{E}$ with natural basis $\{e_1,\dots,e_n\}$ and product given by $e_i^2=e_i$ for any $i=1,\dots,n$.
\end{remark}
We continue this section by extending the following result of \cite{Ba_67} to our context, which shows that supersolvability is a sufficient condition for lower semimodularity.
\begin{proposition} \label{prop:ls}
	Let $\mathcal{E}$ be a finite-dimensional supersolvable (not necessarily evolution) algebra. Then, $\mathcal{E}$ is lower semimodular.
\end{proposition}
\begin{proof}
	First, we prove that every maximal subalgebra has codimension 1. Suppose that $M$ is a maximal subalgebra of $\mathcal{E}$. We use induction on $\dim{\mathcal{E}}$. Let $I$ be a one-dimensional ideal of $\mathcal{E}$. If $I\subseteq M$, then $M/I$ has codimension 1 in $\mathcal{E}/I$, and $\codim{M}=1$. If $I\nsubseteq M$, then $\mathcal{E}=\langle M, I\rangle=M+I$ and $\codim{\mathcal{M}}=\dim{I}=1$ too.
	
	Finally, the result follows from Proposition~\ref{prop:equiv_ls}.
\end{proof}

Next, we characterise supersolvability within solvable evolution algebras.
\begin{theorem}\label{th:ssolv}
	Let $\mathcal{E}$ be a solvable evolution algebra over any field $\mathbb{K}$. $\mathcal{E}$ is supersolvable if and only if every quotient by an ideal is degenerate or has a basic ideal isomorphic to $\mathcal{E}_k(\lambda_1,\dots,\lambda_k)$ for some $\lambda_1,\dots,\lambda_k\in\mathbb{K}$ and $k\leq \dim{\mathcal{E}}$.
\end{theorem}
\begin{proof}
	As supersolvability is closed under quotients, it is enough to show that $\mathcal{E}$ has a one-dimensional ideal if and only if $\mathcal{E}$ is degenerate or has a basic ideal isomorphic to $\mathcal{E}_k(\lambda_1,\dots,\lambda_k)$ for some $\lambda_1,\dots,\lambda_k\in\mathbb{K}$.
	
	Let $\mathcal{E}$ be a solvable non-degenerate evolution algebra. Suppose that there exists a one-dimensional ideal
	\[J=\spa\left\{\sum_{i\in\Lambda}\alpha_ie_i\right\}, \text{ with }\Lambda\subset\{1,\dots,n\},\abs{\Lambda}>1 \text{ and }\alpha_i\neq0 \text{ for any }i\in\Lambda;\]
	that is, $e_i^2$ is collinear with the vector $\sum_{i\in\Lambda}\alpha_ie_i$ for any $i\in\Lambda$ and $\big(\sum_{i\in\Lambda}\alpha_ie_i\big)^2=0$. Equivalently, $I=\oplus_{i\in\Lambda}e_i$ is a solvable basic ideal such that $\dim{I^2}=\dim{J}=1$. By \cite[Proposition 2.5]{CGOT_13}, $I$ is isomorphic to $\mathcal{E}_{\abs{\Lambda}}(\lambda_1,\dots,\lambda_{\abs{\Lambda}})$ for certain $\lambda_1,\dots,\lambda_{\abs{\Lambda}}\in\mathbb{K}$.
\end{proof}
\begin{remark}
	We cannot remove the condition of the ideal isomorphic to $\mathcal{E}_k(\lambda_1,\dots,\lambda_k)$ being basic. Indeed,
	let $\mathcal{E}$ be an evolution algebra with natural basis $\{e_1,e_2,e_3\}$ and product given by $e_1^2=e_1+e_2$, $e_2^2=e_3$ and $e_3^2=-e_1-e_2-e_3$. Notice that $I=\spa\{f_1=e_1+e_2,f_2=e_3\}$ is an evolution ideal such that $f_1^2=f_1+f_2$ and $f_2^2=-f_1-f_2$, then $I\cong\mathcal{E}_2(1,-1)$. Nevertheless, $\mathcal{E}$ does not have one-dimensional ideals, so it is not supersolvable.
\end{remark}

\section{Subalgebra lattice of nilpotent evolution algebras}\label{nilpotentcase}

The main purpose of this section is to study the subalgebra lattice of a nilpotent evolution algebra. The first part of the section deals with distributivity, while the second is devoted to modularity.

First, we characterise distributivity in terms of the index of nilpotency.

\begin{theorem}\label{th:equiv_dist}
	Let $\mathcal{E}$ be an $n$-dimensional nilpotent evolution algebra. Then the following are equivalent:
	\begin{enumerate}[\rm (i)]
		\item $\mathcal{E}$ has maximum index of nilpotency;
		\item its subalgebra lattice is a chain of length $n$;
		\item $\mathcal{E}$ is distributive; and
		\item $\mathcal{E}$ is spanned by the principal powers of an element.
	\end{enumerate}
\end{theorem}
\begin{proof}
	Consider a natural basis $B=\{e_1,\dots,e_n\}$ such that the structure matrix is strictly upper triangular.
	
	(i)$\implies$(ii). If $\mathcal{E}$ is a nilpotent evolution algebra with maximum index of nilpotency,
	by Corollary~\ref{cor:sub_chain}, all subalgebras are $\spa\{e_k,\dots,e_n\}$,  $k=1,\dots,n$. Consequently, its subalgebra lattice is a chain of length $n$.
	
	(ii)$\implies$(iii). Obvious.
	
	(iii)$\implies$(i). Suppose that  $a_{12}a_{23}\dots a_{(n-1)n}=0$ and denote by $a_{k(k+1)}$ the last zero structure constant. Since $a_{k(k+1)}=0$ it holds that $e_k^2\in\spa\{e_{k+2},\dots,e_n\}$. Hence, 
	
	$U=\spa\{e_k,e_{k+2},\dots,e_n\}$, $V=\spa\{e_{k+1},e_{k+2},\dots,e_n\}$ and $W=\spa\{e_k + e_{k+1},e_{k+2},\dots,e_n\}$ are different subalgebras. However,
	\[U=\langle U,V\cap W\rangle\neq\langle U, V\rangle\cap\langle U, W\rangle=\spa\{e_k,\dots,e_n\},\] and its subalgebra lattice is not distributive, a contradiction.

	(iv)$\implies$(i). Consider an arbitrary element $u=\mu_1e_1+\dots+\mu_ne_n$. Then, we have that
	\begin{align*}
		u^{2}&=\mu_1^2a_{12}e_2+v_{3},\text{ with }v_{3}\in\spa\{e_3,\dots,e_n\};\\
		&\:\ \vdots \\
		u^{i}&=\mu_1^2\mu_2\dots\mu_{i-1}a_{12}a_{23}\dots a_{(i-1)i}e_i+v_{i+1},\text{ with } v_{i+1}\in\spa\{e_{i+1},\dots,e_n\};\\
		&\:\ \vdots \\
		u^n&=\mu_1^2\mu_2\dots\mu_{n-1}a_{12}a_{23}\dots a_{(n-1)n}e_n\in\spa\{e_n\}.
	\end{align*}
	Consequently, if $\mathcal{E}=\spa\{u,u^2,\dots,u^n\}$ then $\mu_i, a_{i(i+1)}\neq0$ for any $i=1,\dots,n-1$, what yields the claim.
\end{proof}

The following two examples illustrate how the hypothesis of nilpotency cannot be relaxed to solvability. Example~\ref{ex:1} shows that solvable evolution algebras with maximum index of solvability are not necessarily distributive, and Example~\ref{ex:2} shows that even when they are distributive and supersolvable, their subalgebra lattice is not necessarily a chain.

\begin{example}\label{ex:1}
	Let $\mathcal{E}$ be an evolution algebra with natural basis $\{e_1,e_2,e_3\}$ and multiplication given by $e_1^2=2e_1+2e_2+4e_3$, $e_2^2=2e_1+2e_2$, and  $e_3^2=-e_1^2-e_2^2=-4e_1-4e_2-4e_3$. In fact, it is solvable with maximum index of solvability since $\mathcal{E}^{(2)}=\spa\{e_1+e_2,e_3\}$, $\mathcal{E}^{(3)}=\spa\{e_1+e_2+e_3\}$ and $\mathcal{E}^{(4)}=0$.
	It is easy to show that its subalgebras are the following, and its subalgebra lattice is not distributive.
	
	\vspace{0.2cm}
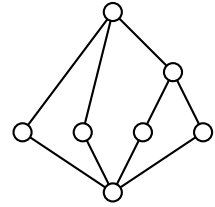
	\begin{minipage}{0.6\textwidth}
		\begin{center}
			\begin{tabular}{| l | l |}
				\hline
				\textbf{Subalg. of dim. $1$} & \textbf{Subalg. of dim. $2$}  \\
				\hline
				$\spa\{e_1+e_2+e_3\}$  & $\spa\{e_1+e_2,e_3\}$ \\
				$\spa\{e_1-e_2+e_3\}$   & \\
				$\spa\{e_1+e_2-e_3\}$ & \\
				$\spa\{e_1-e_2-e_3\}$ & \\
				\hline
			\end{tabular}
		\end{center}
	\end{minipage}\hspace{1cm}
	\begin{minipage}{0.3\textwidth}
		\begin{center}
			\begin{tikzpicture}[scale=0.8]
				\draw[thick] (0,-1.5) -- (-0.5,-0.5);
				\draw[thick] (0,-1.5) -- (-1.5,-0.5);
				\draw[thick] (0,-1.5) -- (1.5,-0.5);
				\draw[thick] (0,-1.5) -- (0.5,-0.5);
				\draw[thick] (1,0.5) -- (0.5,-0.5);
				\draw[thick] (1,0.5) -- (1.5,-0.5);
				\draw[thick] (1,0.5) -- (0,1.5);
				\draw[thick] (-1.5,-0.5) -- (0,1.5);
				\draw[thick] (-0.5,-0.5) -- (0,1.5);
				\draw[fill=white,thick=black](0.5,-0.5) circle [radius=0.15];
				\draw[fill=white,thick=black](1.5,-0.5) circle [radius=0.15];
				\draw[fill=white,thick=black](-0.5,-0.5) circle [radius=0.15];
				\draw[fill=white,thick=black](-1.5,-0.5) circle [radius=0.15];
				\draw[fill=white,thick=black] (0,-1.5) circle [radius=0.15];
				\draw[fill=white,thick=black] (1,0.5) circle [radius=0.15];
				\draw[fill=white,thick=black] (0,1.5) circle [radius=0.15];
			\end{tikzpicture}
		\end{center}
	\end{minipage}
\end{example}

\

\begin{example}\label{ex:2}
	Let $\mathcal{E}$ be an evolution algebra with natural basis $\{e_1,e_2,e_3\}$ and multiplication given by $e_1^2=-e_2^2=e_1+e_2$ and $e_3^2=e_2$. In fact, it is solvable since $\mathcal{E}^{(2)}=\spa\{e_1+e_2,e_2\}=\spa\{e_1,e_2\}$, $\mathcal{E}^{(3)}=\spa\{e_1+e_2\}$ and $\mathcal{E}^{(4)}=0$. Notice that $\mathcal{E}$ is clearly supersolvable. Again, it is easy to show that its subalgebras are the following, and its subalgebra lattice is not a chain.
	
	\begin{minipage}{0.6\textwidth}
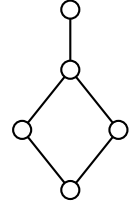
		\begin{center}
			\begin{tabular}{| l | l |}
				\hline
				\textbf{Subalg. of dim. $1$} & \textbf{Subalg. of dim. $2$}  \\
				\hline
				$\spa\{e_1+e_2\}$  & $\spa\{e_1,e_2\}$ \\
				$\spa\{e_1-e_2\}$   & \\
				\hline
			\end{tabular}
		\end{center}
	\end{minipage}
	\hspace{1cm}
	\begin{minipage}{0.3\textwidth}
		\begin{center}
			\begin{tikzpicture} [scale=0.8]
				\draw[thick] (-4,0.5) -- (-4,1.5);	
				\draw[thick] (-4,0.5) -- (-4.8,-0.5);	
				\draw[thick] (-4,0.5) -- (-3.2,-0.5);	
				\draw[thick] (-4,-1.5) -- (-4.8,-0.5);	
				\draw[thick] (-4,-1.5) -- (-3.2,-0.5);	
				\draw[fill=white,thick=black] (-4,-1.5) circle [radius=0.15];
				\draw[fill=white,thick=black](-4.8,-0.5) circle [radius=0.15];
				\draw[fill=white,thick=black](-3.2,-0.5) circle [radius=0.15];
				\draw[fill=white,thick=black](-4,0.5) circle [radius=0.15];
				\draw[fill=white,thick=black](-4,1.5) circle [radius=0.15];
			\end{tikzpicture}
		\end{center}
	\end{minipage}
\end{example}
Next, we apply Theorem~\ref{prop:equiv_mod_qi} to study modularity in nilpotent evolution algebras.
\begin{proposition}\label{prop:lin_comb}
	Let $\mathcal{E}$ be a nilpotent evolution algebra over any field $\mathbb{K}$ of characteristic different from $2$.
	If $\mathcal{E}$ is modular, then an absolute nilpotent element outside the annihilator does not exist.
\end{proposition}
\begin{proof}
	Consider a natural basis $B=\{e_1,\dots,e_n\}$ such that the structure matrix is strictly upper triangular.
	On the contrary, take an absolute nilpotent element $\alpha_{1}e_{1}+\dots+\alpha_{k}e_{k}$ with $\alpha_k\neq0$ and $e_{1},\dots,e_{k}\notin\ann(\mathcal{E})$. Notice that at least two scalars must be non-zero. Therefore, the elements $\pm\alpha_{1}e_{1}\pm\dots\pm\alpha_{k}e_{k}$ are absolute nilpotent too. Hence, the span of each one of these elements is also a subalgebra. Then, consider the next two elements:
	\begin{align*}
		u_1&=\alpha_{1}e_{1}+\dots+\alpha_{k}e_{k},\quad u_1^2=0;\\
		u_2&=\alpha_{1}e_{1}+\dots-\alpha_{k}e_{k},\quad u_2^2=0.
	\end{align*}
	We will prove that $\spa\{u_1\}$ and $\spa\{u_2\}$ are not quasi-ideals. First of all, we claim that $u_1u_2\neq0$. Suppose that $u_1u_2=0$, then
	\begin{align*}
		0=(\alpha_{1}e_{1}^2+\dots+\alpha_{k}e_{k}^2)-(\alpha_{1}e_{1}^2+\dots-\alpha_{k}e_{k}^2)=2\alpha_{k}e_{k}^2
	\end{align*}
	Since the characteristic of the field is different from two and $\alpha_{k}\neq0$, we can conclude that $e_{k}^2=0$ is a contradiction with $e_{k}\notin\ann(\mathcal{E})$. Next, we claim that $\supp(u_1u_2)\cap\supp(u_1)=\supp(u_1u_2)\cap\supp(u_2)=\emptyset$. By hypothesis, we have that
	\[\alpha_{1}e_{1}^2+\dots+\alpha_{k-1}e_{k-1}^2=-\alpha_{k}e_{k}^2\in\spa\{e_l\mid l>k\}\] because the structure matrix is strictly upper triangular without loss of generality. Therefore,
	\[u_1u_2=(\alpha_{1}e_{1}^2+\dots+\alpha_{k-1}e_{k-1}^2)-\alpha_{k}e_{k}^2\in\spa\{e_l\mid l>k\}.\]
	Hence, since $\supp(u_1)=\supp(u_2)\subset\{1,\dots,k\}$ and $\min\{\supp(u_1u_2)\}>k$, their intersection is the empty set, and therefore $u_1u_2\notin \spa\{u_1,u_2\}$.
	
	We conclude that $\dim(\langle \spa\{u_1\},\spa\{u_2\} \rangle)\geqslant3$, and $\spa\{u_1\}$ and $\spa\{u_2\}$ are not quasi-ideals. By Proposition~\ref{prop:equiv_mod_qi}, we are done.
\end{proof}

The following example shows that the converse of the previous result does not hold in general.
\begin{example}
	Let $\mathcal{E}$ be the nilpotent evolution algebra with natural basis $\{e_1,e_2,e_3,e_4,e_5,e_6\}$ and product given by
	$e_1^2=e_4+e_5+e_6, e_2^2=-e_5, e_3^3=-e_6\text{ and } e_4^2=e_5^2=e_6^2=0$. Then, consider the next two subalgebras:
	\[\mathcal{E}_1=\spa\{e_1+e_2,e_4+e_6\}\quad\text{and}\quad\mathcal{E}_2=\spa\{e_1+e_3,e_4+e_5\}.\] But,
	\[\langle\mathcal{E}_1,\mathcal{E}_2\rangle=\spa\{e_1+e_2,e_1+e_3,e_4,e_5,e_6\}\neq\mathcal{E}_1+\mathcal{E}_2.\] Therefore, $\mathcal{E}_1$ and $\mathcal{E}_2$ are not quasi-ideals;  consequently, $\mathcal{E}$ is neither modular nor upper semimodular.
\end{example}
Moreover, the choice of the field gains importance in Proposition~\ref{prop:lin_comb}.
\begin{example}\label{ex:CR_algebra}
	Let $\mathcal{E}$ be the nilpotent evolution algebra with natural basis $\{e_1,e_2,e_3\}$ and product given by $e_1^2=e_2^2=e_3$ and $e_3^2=0$. It is easy to check that its subalgebra lattices as a $\mathbb{C}$-algebra and as an $\mathbb{R}$-algebra are, respectively, the following:
	
	\vspace{0.1cm}
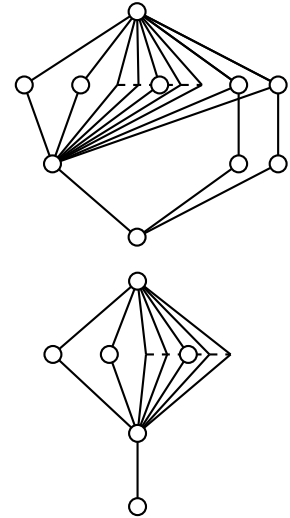
	\begin{minipage}{0.6\textwidth}
		\hspace{-0.2cm}\begin{center}
			\vspace{-0.5cm}
			\begin{tabular}{| l | l |}
				\hline
				\textbf{Subalg. of dim. $1$} & \textbf{Subalg. of dim. $2$}  \\
				\hline
				$\spa\{e_3\}$  & $\spa\{e_1,e_3\}$ \\
				$\spa\{e_1+ie_2\}$   & $\spa\{e_2,e_3\}$  \\
				$\spa\{e_1-ie_2\}$ & $\spa\{e_1+\alpha e_2,e_3\},\alpha\in\mathbb{C}^*$\\
				\hline
			\end{tabular}
		\end{center}
	\end{minipage}
	\begin{minipage}{0.4\textwidth}
		\begin{center}
			\begin{tikzpicture}[scale=0.75]
				\draw[dashed,thick]  (4.65,0.7) -- (6.15,0.7);
				\draw[thick] (6.8,-0.7) -- (5,-2);
				\draw[thick] (7.5,-0.7) -- (5,-2);
				\draw[thick] (3.5,-0.7) -- (5,-2);
				\draw[thick] (5,2) -- (3,0.7);
				\draw[thick] (5,2) -- (4,0.7);
				\draw[thick] (5,2) -- (5.4,0.7);
				\draw[thick] (5,2) -- (4.65,0.7);
				\draw[thick] (5,2) -- (6.15,0.7);
				\draw[thick] (5,2) -- (6.8,0.7);
				\draw[thick] (5,2) -- (7.5,0.7);
				\draw[thick] (5,2) -- (5.025,0.7);
				\draw[thick] (5,2) -- (5.775,0.7);
				\draw[thick] (5,2) -- (6.8,0.7);
				\draw[thick] (5,2) -- (7.5,0.7);
				\draw[thick] (3.5,-0.7) -- (3,0.7);
				\draw[thick] (3.5,-0.7) -- (4,0.7);
				\draw[thick] (3.5,-0.7) -- (5.4,0.7);
				\draw[thick] (3.5,-0.7) -- (4.65,0.7);
				\draw[thick] (3.5,-0.7) -- (6.15,0.7);
				\draw[thick] (3.5,-0.7) -- (6.8,0.7);
				\draw[thick] (3.5,-0.7) -- (7.5,0.7);
				\draw[thick] (3.5,-0.7) -- (5.025,0.7);
				\draw[thick] (3.5,-0.7) -- (5.775,0.7);
				\draw[thick] (6.8,-0.7) -- (6.8,0.7);
				\draw[thick] (7.5,-0.7) -- (7.5,0.7);
				
				\draw[fill=white,thick=black] (5,-2) circle [radius=0.15];
				\draw[fill=white,thick=black](3.5,-0.7) circle [radius=0.15];
				\draw[fill=white,thick=black](6.8,-0.7) circle [radius=0.15];
				\draw[fill=white,thick=black](7.5,0.7) circle [radius=0.15];
				\draw[fill=white,thick=black](6.8,0.7) circle [radius=0.15];
				\draw[fill=white,thick=black](5.4,0.7) circle [radius=0.15];
				\draw[fill=white,thick=black](3,0.7) circle [radius=0.15];
				\draw[fill=white,thick=black](4,0.7) circle [radius=0.15];
				\draw[fill=white,thick=black](7.5,-0.7) circle [radius=0.15];
				\draw[fill=white,thick=black](5,2) circle [radius=0.15];	
			\end{tikzpicture}\hspace{0.1cm}
			
		\end{center}
	\end{minipage}\vspace{0.3cm}
	
	\begin{minipage}{0.6\textwidth}
		\hspace{-0.2cm}\begin{center}
			\vspace{-0.5cm}
			\begin{tabular}{| l | l |}
				\hline
				\textbf{Subalg. of dim. $1$} & \textbf{Subalg. of dim. $2$}  \\
				\hline
				$\spa\{e_3\}$  & $\spa\{e_1,e_3\}$ \\
				& $\spa\{e_2,e_3\}$  \\
				& $\spa\{e_1+\alpha e_2,e_3\},\alpha\in\mathbb{R}^*$\\
				\hline
			\end{tabular}
		\end{center}
	\end{minipage}
	\begin{minipage}{0.4\textwidth}
		\begin{center}\hspace{-0.6cm}
			\begin{tikzpicture}[scale=0.75]
				\draw[dashed,thick]  (4.65,0.7) -- (6.15,0.7);
				\draw[thick] (4.5,-0.7) -- (4.5,-2);
				\draw[thick] (4.5,2) -- (3,0.7);
				\draw[thick] (4.5,2) -- (4,0.7);
				\draw[thick] (4.5,2) -- (5.4,0.7);
				\draw[thick] (4.5,2) -- (4.65,0.7);
				\draw[thick] (4.5,2) -- (6.15,0.7);
				\draw[thick] (4.5,2) -- (5.025,0.7);
				\draw[thick] (4.5,2) -- (5.775,0.7);
				\draw[thick] (4.5,-0.7) -- (3,0.7);
				\draw[thick] (4.5,-0.7) -- (4,0.7);
				\draw[thick] (4.5,-0.7) -- (5.4,0.7);
				\draw[thick] (4.5,-0.7) -- (4.65,0.7);
				\draw[thick] (4.5,-0.7) -- (6.15,0.7);
				\draw[thick] (4.5,-0.7) -- (5.025,0.7);
				\draw[thick] (4.5,-0.7) -- (5.775,0.7);

				\draw[fill=white,thick=black] (4.5,-2) circle [radius=0.15];
				\draw[fill=white,thick=black](4.5,-0.7) circle [radius=0.15];
				\draw[fill=white,thick=black](5.4,0.7) circle [radius=0.15];
				\draw[fill=white,thick=black](3,0.7) circle [radius=0.15];
				\draw[fill=white,thick=black](4,0.7) circle [radius=0.15];
				\draw[fill=white,thick=black](4.5,2) circle [radius=0.15];			
			\end{tikzpicture}
		\end{center}
	\end{minipage}\vspace{0.1cm}
	
	In the first case, $e_1+ie_2$ and $e_1-ie_2$ are absolute nilpotent elements outside the annihilator, and consequently, $\mathcal{E}$ is not modular. However, in the second case, there are no absolute nilpotent elements. In fact, $\mathcal{E}$ is modular since $\spa\{e_3\}$, $\spa\{e_1,e_3\}$ and $\spa\{e_2,e_3\}$ are all quasi-ideals. Moreover, notice that, unlike the first lattice, there are no pentagons in the second one.
\end{example}

\begin{corollary}\label{cor:mod_nilp}
	Let $\mathcal{E}$ be a nilpotent evolution algebra over a quadratically closed field of characteristic different from $2$ such that its annihilator is one-dimensional. Then, $\mathcal{E}$ is modular if and only if $\mathcal{E}$ is distributive.
\end{corollary}
\begin{proof}
	We will show that if $\mathcal{E}$ is not distributive, then it is not modular. Suppose that the structure matrix is strictly upper triangular with at least one zero element in its first upper diagonal (see Theorem~\ref{th:equiv_dist}). With this assumption, $\ann(\mathcal{E})\supseteq\spa\{e_n\}$.
	Let $M_B(\mathcal{E})^{(n,1)}$ be the matrix resulting from deleting the structure matrix's first column and last row.
	Since $M_B(\mathcal{E})^{(n,1)}$ has at least one zero element in its diagonal, its determinant is zero. Then the rows of $M_B(\mathcal{E})^{(n,1)}$, and consequently the first $n-1$ rows of $M_B(\mathcal{E})$, are linearly dependent, that is, there exist scalars $\alpha_1,\dots,\alpha_{n-1}\in\mathbb{K}$, some of them non-zero, such that $\alpha_1e_1^2+\dots+\alpha_{n-1}e_{n-1}^2=0$. As $\mathbb{K}$ is a quadratically closed field, we have that $\sqrt{\alpha_1}e_1+\dots+\sqrt{\alpha_{n-1}}e_{n-1}$ is an absolute nilpotent element. Then, by Proposition~\ref{prop:lin_comb}, we conclude that $\mathcal{E}$ is not modular.
\end{proof}

We conclude this section with an example showing that Corollary~\ref{cor:mod_nilp} does not necessarily hold if the dimension of the annihilator is greater than $1$.
\begin{example}
	Let $\mathcal{E}$ be the nilpotent evolution $\mathbb{C}$-algebra with natural basis $\{e_1,e_2,e_3\}$ and product given by $e_1^2=e_2$ \ and \ $e_2^2=e_3^2=0$. Notice that $\ann(\mathcal{E})=\spa\{e_2,e_3\}$ has dimension two. It is easy to see that its subalgebras are the following, and  its subalgebra lattice is modular but not distributive.
	
	\begin{minipage}{0.6\textwidth}
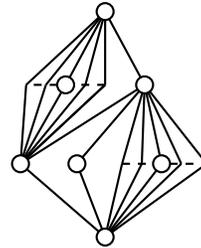
		\begin{center}
			\begin{tabular}{| l | l |}
				\hline
				\textbf{Subalg. of dim. $1$} & \textbf{Subalg. of dim. $2$}  \\
				\hline
				$\spa\{e_2\}$  & $\spa\{e_1+\alpha e_3,e_2\}$, $\alpha\in\mathbb{C}$ \\
				$\spa\{e_3\}$   & $\spa\{e_2,e_3\}$  \\
				$\spa\{e_2+\alpha e_3\}$, $\alpha\in\mathbb{C}^*$ & \\
				\hline
			\end{tabular}
		\end{center}
	\end{minipage}
	\hspace{1cm}
	\begin{minipage}{0.3\textwidth}
		\begin{center}
			\begin{tikzpicture}[scale=0.75]
				\draw[dashed,thick]  (0.3,-0.7) -- (1.8,-0.7);
				\draw[dashed,thick]  (0,0.7) -- (-1.4,0.7);
				\draw[thick] (0,0.7) -- (-1.5,-0.7);
				\draw[thick] (-1.4,0.7) -- (0,2);
				\draw[thick] (0,0.7) -- (0,2);
				\draw[thick] (-0.35,0.7) -- (0,2);
				\draw[thick] (-1.05,0.7) -- (0,2);
				\draw[thick] (-1.4,0.7) -- (-1.5,-0.7);
				\draw[thick] (0,0.7) -- (-1.5,-0.7);
				\draw[thick] (-0.35,0.7) -- (-1.5,-0.7);
				\draw[thick] (-1.05,0.7) -- (-1.5,-0.7);
				\draw[thick] (-0.7,0.7) -- (0,2);
				\draw[thick] (0.7,0.7) -- (0,2);
				\draw[thick] (-0.7,0.7) -- (-1.5,-0.7);
				\draw[thick] (0,-2) -- (-0.5,-0.7);
				\draw[thick] (0,-2) -- (-1.5,-0.7);
				\draw[thick] (0,-2) -- (0.3,-0.7);
				\draw[thick] (0,-2) -- (0.65,-0.7);
				\draw[thick] (0,-2) -- (1.35,-0.7);
				\draw[thick] (0,-2) -- (1.8,-0.7);
				\draw[thick] (0,-2) -- (1,-0.7);	
				\draw[thick] (0.7,0.7) -- (-0.5,-0.7);
				\draw[thick] (0.7,0.7) -- (-1.5,-0.7);
				\draw[thick] (0.7,0.7) -- (0.3,-0.7);
				\draw[thick] (0.7,0.7) -- (0.65,-0.7);
				\draw[thick] (0.7,0.7) -- (1.35,-0.7);
				\draw[thick] (0.7,0.7) -- (1.8,-0.7);
				\draw[thick] (0.7,0.7) -- (1,-0.7);
				\draw[fill=white,thick=black](1,-0.7) circle [radius=0.15];
				\draw[fill=white,thick=black](-0.5,-0.7) circle [radius=0.15];
				\draw[fill=white,thick=black](-1.5,-0.7) circle [radius=0.15];
				\draw[fill=white,thick=black] (0,-2) circle [radius=0.15];
				\draw[fill=white,thick=black] (-0.7,0.7) circle [radius=0.15];
				\draw[fill=white,thick=black] (0.7,0.7) circle [radius=0.15];
				\draw[fill=white,thick=black] (0,2) circle [radius=0.15];
			\end{tikzpicture}
		\end{center}
	\end{minipage}
\end{example}

\section{Subalgebra lattice of solvable evolution algebras}\label{solvablecase}

The last section of the paper is devoted to the study of modularity in the context of solvable non-nilpotent evolution algebras.
Our first objective is to characterise distributivity and modularity in the evolution algebras described in \eqref{def:solv_1} and Definition~\ref{def:solv_2}.
\begin{lemma}\label{lem:solv_mod}
	Let $\mathcal{E}$ be a solvable non-nilpotent evolution algebra with one-dimensional derived subalgebra over any field of characteristic different from $2$. Then, it is distributive if and only if it is modular and if and only if $n=2$, i.e. $\mathcal{E}\cong\mathcal{E}_2(1,-1)$.
\end{lemma}
\begin{proof}
	By Lemma~\ref{lem:clas_solv_anyfield}, if $n=2$ then $\mathcal{E}\cong\mathcal{E}_2(1,-1)$, whose subalgebras are $\spa\{e_1-e_2\}$ and $\spa\{e_1+e_2\}$ and give rise to a distributive lattice.
	Suppose that $n>2$. By \cite[Proposition~2.5]{CGOT_13}, $\mathcal{E}$ is isomorphic to $\mathcal{E}_k(\lambda_1,\dots,\lambda_n)$ for certain $\lambda_1,\dots,\lambda_n\in\mathbb{K}$ and $k\in\mathbb{N}$. Following \cite[Remark~2.6]{CGOT_13}, we will assume that $\lambda_i\neq0$ for some $1\leq i\leq k$, in fact, for at least two indices. We consider the following cases:
	\begin{enumerate}[\rm (i)]
		\item If $k\geq3$, suppose, without loss of generality, that $\lambda_1,\dots,\lambda_l\neq0$ for some $l\leq k$.  Necessarily $l\geq2$ since $\mathcal{E}$ is non-nilpotent. So, let us define the following elements of $\mathcal{E}$: \[u=-e_{1}+e_{2}+\dots+e_{l}\quad\text{ and }\quad v=e_{1}+\dots+e_{l-1}-e_{l}.\]
		Note that both $u$ and $v$ span a one-dimensional subalgebra since \[u^2=v^2=(e_1+\dots+e_k)\sum_{j=1}^{l}\lambda_j=(e_1+\dots+e_k)\sum_{j=1}^{k}\lambda_j=0.\]
		Nevertheless, we claim that $\spa\{u\}$ is not a quasi-ideal. To prove it, we consider the following two cases:
		\begin{enumerate}[\rm (a)]
			\item If $uv\neq0$, then $uv=k(e_1+\dots+e_k), k\in\mathbb{K}^*$. So $\langle u,v\rangle=\spa\{u,v\}+\mathcal{E}^2$;  consequently, $\spa\{u\}$ is not a quasi-ideal.
			\item If $uv=0$,  it is easy to deduce that $e_{1}^2+e_{l}^2=0$. So $\langle u,e_{1}+e_{l}\rangle=\spa\{u,e_{1}+e_{l}\}+\mathcal{E}^2$ because $u(e_{1}+e_{l})\neq0$. Consequently, $\spa\{u\}$ is not a quasi-ideal.
		\end{enumerate}
		\item Suppose that $k=2$. Then, without loss of generality, $\spa\{e_1,e_2\}$ is a subalgebra and,
		by Lemma~\ref{lem:clas_solv_anyfield}, we can assume that $e_1^2=-e_2^2=e_1+e_2$. If $e_3^2\neq0$, then it is easy to prove that there exists a scalar $\alpha\in\mathbb{K}^*$ such that $\spa\{e_i+\alpha e_3\}$ is a subalgebra for $i=1$ or $i=2$ because $e_3^2\in\mathcal{E}^2=\spa\{e_1+e_2\}$. Then, $\spa\{e_1-e_2\}$ and $\spa\{e_i+\alpha e_3\}$ are subalgebras which are not quasi-ideals. On the other hand, if $e_3^2=0$, then it is enough to consider the subalgebras $\spa\{e_1-e_2\}$ and $\spa\{e_1+e_2+e_3\}$, which are not quasi-ideals.
	\end{enumerate}
	In conclusion, when $n>2$, there is always a subalgebra that is not a quasi-ideal. So, by Theorem~\ref{prop:equiv_mod_qi}, $\mathcal{E}$ is not modular, and consequently, it is not distributive either.
\end{proof}
\begin{corollary}\label{cor:sum_direct_solv}
	The direct sum of solvable non-nilpotent evolution algebras over any field of characteristic different from $2$ of the previous type is not modular.
\end{corollary}
\begin{proof}
	If any of the addends has dimension greater than two then, by Lemma~\ref{lem:solv_mod}, it is not modular. Therefore, the only possibility would be that all the addends have dimension two. Nevertheless, we claim that it would not be modular either. 
	Let $\mathcal{E}_1=\spa\{e_1,e_2\}$ and $\mathcal{E}_2=\spa\{e_3,e_4\}$ be the first two addends of the direct sum. 
	Again, by Lemma~\ref{lem:clas_solv_anyfield}, their multiplication is given by $e_1^2=-e_2^2=e_1+e_2$ and $e_3^2=-e_4^2=e_3+e_4$, respectively.
	Finally, it is enough to consider the subalgebras $U=\spa\{e_1-e_2+e_3+e_4\}$ and $V=\spa\{e_1+e_2+e_3-e_4\}$ because
	\begin{align*}
		\langle U,V\rangle &=\spa\{e_1-e_2+e_3+e_4,e_1+e_2+e_3-e_4,e_1+e_2+e_3+e_4\}\\&=\spa\{e_1+e_3,e_2,e_4\}\neq U+V,
	\end{align*}
	so they are not quasi-ideals.
\end{proof}
Finally, we establish the following theorem,  generalising the previous two cases when the field is quadratically closed.
\begin{theorem}\label{th:clas_mod}
	Let $\mathcal{E}_1,\dots,\mathcal{E}_r$ be solvable non-nilpotent evolution algebras of the type described in \eqref{def:solv_1} over a quadratically closed field $\mathbb{K}$ of characteristic different from $2$ and let $\mathcal{E}$ be a solvable non-nilpotent evolution algebra of the family $\mathcal{F}(\mathcal{E}_1,\dots,\mathcal{E}_r)$
	(according to Definition~\ref{def:solv_2}). Then, the following statements are equivalent:
	\begin{enumerate}[\rm (i)]
		\item $\mathcal{E}$ is distributive;
		\item $\mathcal{E}$ is modular and, consequently, upper semimodular;
		\item $r=1$, $\mathcal{E}_1\cong\mathcal{E}_2(1,-1)$ and the quotient $\mathcal{E}/\mathcal{E}_1^2$ is a nilpotent evolution algebra with maximum index of nilpotency; and
		\item $\mathcal{E}$ has maximum index of solvability.
	\end{enumerate}
	In this case, there exists a natural basis such that the structure matrix is
	\[M_{B}(\mathcal{E})=\begin{pmatrix*}[l] A_{2\times 2}&0_{2\times(n-2)}\\C_{(n-2)\times 2}&L_{(n-2)\times(n-2)} \end{pmatrix*},\] where $A=\big(\begin{smallmatrix*}[r]
		1 & 1\\
		-1 & -1
	\end{smallmatrix*}\big)$, $L$ is strictly lower triangular with no zeros in its first lower diagonal and $M_B(\mathcal{E})$ has rank $n-1$.
\end{theorem}
\begin{proof}\hfill
	
	(i)$\implies$(ii). Obvious.
	
	(ii)$\implies$(iii). Suppose that the structure of $\mathcal{E}$ is not the one described in statement $\rm{(iii)}$. First, if $r>1$ or $\dim{\mathcal{E}_1}>2$,  $\mathcal{E}$ is not modular by Corollary~\ref{cor:sum_direct_solv} and 
	Lemma~\ref{lem:solv_mod}, respectively.
	Secondly, we claim that $\mathcal{E}$ is not modular if $\mathcal{E}/\mathcal{E}_1^2$ does not have maximum index of nilpotency.
	We must consider the following two cases:
	\begin{enumerate} [\rm (a)]
		\item If the annihilator of $\mathcal{E}/\mathcal{E}_1^2$ is one-dimensional, it is not upper semimodular (or modular)
		by Corollary~\ref{cor:mod_nilp}. Then, there exists a subalgebra $U\subset\mathcal{E}/\mathcal{E}_1^2$ which is not upper semimodular. Denote by $\pi$ the usual projection to the quotient. Since $\mathcal{E}_1^2$ is ideal, 
		by~\cite[Lemma~1.2]{To_86}, $\mathcal{E}_1^2\subset V=\pi^{-1}(U)\subset\mathcal{E}$  is not upper semimodular either. Then, $\mathcal{E}$ is not upper semimodular.
		
		\item If the annihilator of $\mathcal{E}/\mathcal{E}_1^2$ has dimension greater than 1, then there exists an index $i\in\{3,\dots,n\}$ such that $e_i^2\in\spa\{e_1+e_2\}=\mathcal{E}_1^2$. Assume $i=3$.  Next, we show that $\spa\{e_1-e_2\}$ is not a quasi-ideal.
		\begin{enumerate}[\rm (1)]
			\item If $e_3^2=0$, then $\spa\{e_1+e_2+e_3\}$ is a subalgebra but $\langle e_1-e_2,e_1+e_2+e_3\rangle=\spa\{e_1,e_2,e_3\}$.
			\item If $e_3^2\neq0$, then there exists a scalar $\alpha\in\mathbb{K}^*$ such that $\spa\{e_1+\alpha e_3\}$ is a subalgebra but $\langle e_1-e_2,e_1+\alpha e_3\rangle=\spa\{e_1,e_2,e_3\}$.
		\end{enumerate}
	\end{enumerate}
	
	(iii)$\implies$(i). If $e_3^2\notin\spa\{e_1-e_2\}$, following the argument used in Corollary~\ref{cor:sub_chain}, it is easy to prove that all its subalgebras are exactly $\spa\{e_1+e_2\}$, $\spa\{e_1-e_2\}$ and those of the form $\spa\{e_1,\dots,e_k\}$ with $k=2,\dots,n$. Consequently, its subalgebra lattice is almost a chain, actually, a chain with a rhombus at the bottom, which is distributive.
	
	Nevertheless, if $e_3^2\in\spa\{e_1-e_2\}$, there exist more subalgebras than in the previous case. In fact, $\spa\{e_1+e_2\}$, $\spa\{e_1-e_2\}$, $\spa\{e_1,e_2\}$ and $\spa\{e_1-e_2,e_3\}$ are exactly the subalgebras contained in $\spa\{e_1,e_2,e_3\}$. Analogously, if $e_4^2\in\spa\{e_1-e_2,e_3\}$, then $\spa\{e_1-e_2,e_3,e_4\}$ and $\spa\{e_1,e_2,e_3\}$ are (together with the previous five) the subalgebras contained in $\spa\{e_1,e_2,e_3,e_4\}$. Following this argument, we deduce that its subalgebra lattice is a chain with a succession of rhombuses linked by an edge at the bottom, which is also distributive.
	
	(iii)$\Longleftrightarrow$(iv). It is straightforward from Proposition~\ref{prop:caract_mis}.
\end{proof}

As we have seen in Theorems~\ref{th:equiv_dist} and \ref{th:clas_mod}, evolution algebras with maximum index of solvability have the best properties. Then, we now focus on this case, looking at their modularity and its relation with supersolvability. Nevertheless, we first establish the following technical lemmas, which will be used later.
\begin{lemma}\label{lem:der_ser}
	Let $\mathcal{E}$ be a solvable evolution algebra with index of solvability $n$ such that $\mathcal{E}^{(m)}$, with $m<n$, is an ideal. Then $\mathcal{E}/\mathcal{E}^{(m)}$ is an evolution algebra with index of solvability $m$.
\end{lemma}
\begin{proof}
	It is enough to realise that the terms of the derived series of $\mathcal{E}/\mathcal{E}^{(m)}$ are exactly 
	\[\left(\mathcal{E}/\mathcal{E}^{(m)}\right)^{(l)}=\mathcal{E}^{(l)}/\mathcal{E}^{(m)},\quad l=1,\dots,m. \qedhere\]
\end{proof}

\begin{lemma}\label{lem:ann}
	Let $\mathcal{E}$ be an $n$-dimensional solvable evolution algebra with maximum index of solvability such that $e_n^2=-e_1^2-\dots-e_m^2$ and let $u\in\mathcal{E}$ be an absolute nilpotent element. Then, $\ann_{\mathcal{E}}(u)\coloneqq\{x\in\mathcal{E}\colon xu=0\}=\spa\{u,e_{m+1},\dots,e_n\}$.
\end{lemma}
\begin{proof}
	By Proposition~\ref{prop:cant_sub} and Remark~\ref{rem:mis}, every absolute nilpotent element is a multiple of $\pm e_1\pm\dots\pm e_{m}\pm e_n$. Then, if we consider an arbitrary element $v=\mu_1e_1+\dots+\mu_ne_n$ we have that
	\begin{align*}
		v\in\ann_{\mathcal{E}}(u)&\Longleftrightarrow vu=(\mu_1e_1+\dots+\mu_ne_n)u=(\mu_1e_1+\dots+\mu_me_m+\mu_n e_n)u=0\\&\Longleftrightarrow\mu_1e_1+\dots+\mu_me_m+\mu_ne_n\in\langle u\rangle\text{ or }\mu_1,\dots,\mu_m,\mu_n=0,
	\end{align*}
	what yields the claim. Just notice that the second equivalence is a consequence of the fact that $e_n^2+e_1^2+\dots+e_m^2=0$ is the only linear dependence relationship in the set $\{e_1^2,\dots,e_n^2\}$.
\end{proof}

Next, we establish a necessary (but non-sufficient) condition for a solvable evolution algebra with maximum index of solvability to be modular. 

\begin{proposition}\label{prop:cond_nec_mod}
	Let $\mathcal{E}$ be an $n$-dimensional solvable evolution algebra with maximum index of solvability over any field of characteristic different from $2$. If $\mathcal{E}$ is modular, then $\mathcal{E}^{(n)}$ or $\mathcal{E}^{(n-1)}$ is a basic ideal.
\end{proposition}
\begin{proof}
	By Proposition~\ref{prop:cant_sub} and Remark~\ref{rem:mis}, there exists a natural basis such that $e_n^2=-e_1^2-\dots-e_m^2$, with $m\in\mathbb{N}$, $m<n$. If $\mathcal{E}$ is degenerate, then $\mathcal{E}^{(n)}=\spa\{e_n\}$ is already a basic ideal. Otherwise, suppose on the contrary that $\mathcal{E}^{(n-1)}$ is not a basic ideal. We distinguish the following two cases:
	
	\begin{enumerate}[\rm (i)]
		\item if $\mathcal{E}^{(n-1)}$ has two different subalgebras, say $V=\spa\{v\}$ and $W=\spa\{w\}$, then necessarily $m>1$. Hence, by Proposition~\ref{prop:cant_sub} and Lemma~\ref{lem:sub_2}, there exists at least a one-dimensional subalgebra, $U=\spa\{u\}$, which is not contained in $\mathcal{E}^{(n-1)}$.
		Take $U$ so that the number $l\in\mathbb{N}$ is such that $U\subset\mathcal{E}^{(l)}$ but $U\nsubseteq\mathcal{E}^{(l+1)}$ is maximum. By Proposition~\ref{prop:equiv_mod_qi}, every subalgebra is a quasi-ideal,
		then $U+V=\spa\{u,v\}$ and $U+W=\spa\{u,w\}$ are subalgebras. 
		On the one hand, as $(U+V)^2,(U+W)^2\subset\mathcal{E}^{(l+1)}$,
		$(U+V)^{(3)},(U+W)^{(3)}=0$ and the multiples of $v$ and $w$ are the only absolute nilpotent elements of $\mathcal{E}^{(l+1)}$, it holds that $0\neq uv\in\  V$ and  $0\neq u w\in W$.
		On the other hand, it holds that either $0\neq vw\in V$ or $0\neq v w\in W$.
		Consider the first case; the other is analogue. As $0\neq uv\in V$ and $0\neq vw\in V$, 
		there exist scalars $\alpha,\beta\in\mathbb{K}^*$ such that $v(\alpha u+\beta w)=0$. Then, by Lemma~\ref{lem:ann},
		it holds that $\alpha u+\beta w\in V$, which contradicts the linear independence of $\{u,v,w\}$. 
		
		\item if $\mathcal{E}^{(n-1)}$ only has one subalgebra, $\mathcal{E}^{(n)}=\spa\{v\}$, by the proof of Lemma~\ref{lem:sub_2} there exists an element $w\in\mathcal{E}$ such that $\mathcal{E}^{(n-1)}=\spa\{v,w\}$ and $v^2=0,w^2=v$ and $vw=0$. Again, we take $U$ as in the previous case, and we consider the subalgebra $\spa\{u,v\}$, where $0\neq uv\in V$. As a consequence of Lemma~\ref{lem:ann}, it holds that $\supp(w)\subset\{m+1,\dots,n-1\}$ and, since $0\neq uv$, $w^2\in \mathcal{E}^{(n)}$, there exists a contradiction with the linear independence of $\{e_1^2,\dots,e_{n-1}^2\}$.
	\end{enumerate}
	As in both cases we get a contradiction, the result follows.
\end{proof}
\begin{remark}
	Notice that the previous necessary condition for modularity is not equivalent to the fact that there exists a reordering of the natural basis such that $e_n^2=-e_1^2$. Let $\mathcal{E}$ be an evolution algebra with natural basis $\{e_1,e_2,e_3\}$ and product given by $e_1^2=-e_3^2=e_2$ and $e_2^2=e_1+e_3$. In fact, $e_3^2=-e_1^2$ but $\mathcal{E}^2=\spa\{e_2,e_1+e_3\}$, which is not a basic ideal.
\end{remark}
\begin{corollary}\label{cor:cond_nec_mod}
	Let $\mathcal{E}$ be a solvable evolution algebra with maximum index of solvability over any field of characteristic different from $2$. If $\mathcal{E}$ is modular, then the derived series does not have two consecutive terms that are not basic ideals.
\end{corollary}
\begin{proof}
	As modularity is closed under quotients, the result follows straightforward from Proposition~\ref{prop:cond_nec_mod}, Lemma~\ref{lem:der_ser} and Remark~\ref{rem:quo_bas_i}.
\end{proof}

\begin{remark}
	The converse of Corollary~\ref{cor:cond_nec_mod} is not necessarily true. Let $\mathcal{E}$ be an evolution algebra over a field of characteristic different from $2$ with natural basis $\{e_1,e_2,e_3,e_4\}$ and product given by $e_1^2=-e_2^2=e_1+e_2$, $e_3^2=e_1+e_3+e_4$ and $e_4^2=-e_2+e_3+e_4$. In fact, it is solvable with maximum index of solvability since $\mathcal{E}^2=\spa\{e_1,e_2,e_3+e_4\}$, $ \mathcal{E}^{(3)}=\spa\{e_1,e_2\}$, $\mathcal{E}^{(4)}=\spa\{e_1+e_2\}$ and $\mathcal{E}^{(5)}=0$. Moreover, notice that every other term of the derived series is a basic ideal. However, $\mathcal{E}_1=\spa\{e_1-e_2,e_3+e_4\}$ and $\mathcal{E}_2=\spa\{e_1-e_2,e_3-e_4\}$ are not quasi-ideals since
	\[\langle\mathcal{E}_1,\mathcal{E}_2\rangle=\spa\{e_1,e_2,e_3,e_4\}\neq\spa\{e_1-e_2,e_3,e_4\}=\mathcal{E}_1+\mathcal{E}_2.\]
\end{remark}

The next result  proves that this necessary condition is equivalent to supersolvability. Then, although it is not a sufficient condition for modularity, by Proposition~\ref{prop:ls}, it is indeed sufficient for lower semimodularity.  

\begin{theorem}\label{th:sup_mat_der}
	Let $\mathcal{E}$ be a solvable $n$-dimensional evolution algebra with maximum index of solvability over any field of characteristic different from $2$. Then, the following statements are equivalent:
	\begin{enumerate}[\rm (i)]
		\item $\mathcal{E}$ is supersolvable;
		\item there exists a natural basis $B$ such that the structure matrix is block lower triangular with zeros or blocks $\bigl(\begin{smallmatrix*}[r]1 & 1\\ -1 & -1\end{smallmatrix*}\bigr)$ in its diagonal and rank $n-1$; and
		\item  the derived series does not have two consecutive terms that are not basic ideals.
	\end{enumerate}
	\begin{proof}
		Before starting the proof, just notice that the structure matrix of a quotient by a basic ideal can be seen as the result of removing the corresponding rows and columns in the original one. Then, by Lemma~\ref{lem:der_ser}, item (ii) is preserved under quotients by basic ideals.
		
		(i)$\implies$(ii). As $\mathcal{E}$ has maximum index of solvability and is supersolvable, by Theorem~\ref{th:ssolv}, we have that $\mathcal{E}$ is degenerate or there exists a basic ideal isomorphic to $\mathcal{E}_2(1,-1)$. Then, the result follows from the fact that supersolvability is closed under quotients.
		
		(ii)$\implies$(iii). We will use induction on $\dim{\mathcal{E}}$. Let $M_B(\mathcal{E})=(a_{ij})$ be a structure matrix as described. 
		First, if $a_{11}=0$, then there clearly exists a basic ideal $\mathcal{E}^{(n)}=\spa\{e_1\}$. Otherwise, the multiples of $e_1+e_2$ and $e_1-e_2$ are the only absolute nilpotent elements. Then, we claim that $\mathcal{E}^{(n-1)}=\spa\{e_1+e_2\}$. On the contrary, assume that either $e_1+e_2\notin\mathcal{E}^{(n-1)}$ or $e_1-e_2\notin\mathcal{E}^{(n-1)}$.
		Suppose, without loss of generality, that $e_1-e_2\notin\mathcal{E}^{(n-1)}$.
		Consequently, $\mathcal{E}^{(n-1)}$ only has a one-dimensional subalgebra, $\mathcal{E}^{(n)}=\spa\{e_1+e_2\}$.
		Then, reasoning as in the proof of Proposition~\ref{prop:cond_nec_mod}, we get a contradiction with the fact that $\rank\big(M_B(\mathcal{E})\big)=n-1$. Then, we have just proved that $\mathcal{E}^{(n)}$ or $\mathcal{E}^{(n-1)}$ is a basic ideal.
		Denote by $I$ this basic ideal. Next, by Lemma~\ref{lem:der_ser}, 
		the induction hypothesis can be applied, and  the derived series of $\mathcal{E}/I$ 
		does not have two consecutive terms that are not basic ideals. 
		Finally, using Remark~\ref{rem:quo_bas_i}, the result follows.
		
		(iii)$\implies$(i).     
		We will again use induction on $\dim{\mathcal{E}}$. By hypothesis, $\mathcal{E}^{(n)}$ or $\mathcal{E}^{(n-1)}$ is a basic ideal. Suppose that $\mathcal{E}^{(n)}$ is basic. Then, consider the quotient $\mathcal{E}/\mathcal{E}^{(n)}$, which is a solvable evolution algebra with maximum index of solvability by Lemma~\ref{lem:der_ser}. Then, by the induction hypothesis, $\mathcal{E}/\mathcal{E}^{(n)}$ is supersolvable. As $\mathcal{E}^{(n)}$ is a one-dimensional ideal, $\mathcal{E}$ is clearly supersolvable. Otherwise, reasoning the same $\mathcal{E}/\mathcal{E}^{(n-1)}$ is supersolvable;
		the result follows as  $\mathcal{E}^{(n-1)}$ is supersolvable by Lemma~\ref{lem:clas_solv_anyfield}.     
	\end{proof}
\end{theorem}
For our last result, we will suppose, without loss of generality, that the structure matrix $(a_{ij})_{i,j=1}^n$ of a modular evolution algebra with maximum index of solvability over any field of characteristic different from $2$ is as described
in Theorem~\ref{th:sup_mat_der}. Moreover, define $\varLambda$ as the subset of pairs $(i,i+1)$, which corresponds with the blocks $\bigl(\begin{smallmatrix*}[r]1 & 1\\ -1 & -1\end{smallmatrix*}\bigr)$ in the structure matrix, that is,
\[
\varLambda=\left\{(i,i+1)\colon\bigl(\begin{smallmatrix*}[l] a_{ii} & a_{i(i+1)}\\ a_{(i+1)i} & a_{(i+1)(i+1)}\end{smallmatrix*}\bigr)=\bigl(\begin{smallmatrix*}[r]1 & 1\\ -1 & -1\end{smallmatrix*}\bigr)\right\}.
\]
Indeed, as reasoned in the proof of Theorem~\ref{th:clas_mod} it is easy to check that all subalgebras can be written as $K+\spa\{e_{k_1}-e_{k_1+1},e_{k_2}-e_{k_2+1},\dots,e_{k_l}\pm e_{k_l+1}\}$ for certain $k_1<\dots<k_l$, $(k_1,k_1+1),(k_2,k_2+1),\dots,(k_l,k_l+1)\in\varLambda$ and $K$ is a subspace spanned by elements of the natural basis.
In addition, also denote by $\pi_{ij}$ the linear projection to the subspace $\spa\{e_i,e_j\}$ for any $i\neq j$, $i,j=1,\dots,n$. 
\begin{corollary}
	Let $\mathcal{E}$ be a solvable evolution algebra with maximum index of solvability over any field of characteristic different from $2$. Then, $\mathcal{E}$ is modular if and only if the derived series does not have two consecutive terms that are not basic ideals and there does not exist a subalgebra $U=K+\spa\{e_i-e_{i+1},e_j\pm e_{j+1}\}$ where $i<j$, $(i,i+1),(j,j+1)\in\varLambda$ and $K$ is a subspace spanned by elements of the natural basis different from $e_i$, $e_{i+1}$, $e_j$ and $e_{j+1}$; and such that $\pi_{i(i+1)}(e_j)$ or $\pi_{i(i+1)}(e_{j+1})\notin\langle e_i-e_{i+1}\rangle$.
\end{corollary}
\begin{proof}
	First, we prove the sufficiency. By hypothesis, every subalgebra of $\mathcal{E}$ can be written as $\spa\{e_1,\dots,e_r,e_{r+1}\pm e_{r+2}\}$ or $\spa\{e_1,\dots,e_r,e_{r+1}-e_{r+2},e_{r+3},\dots,e_m\}$ for some $r,m\geq0$. Then, as all these subalgebras are quasi-ideals, $\mathcal{E}$ is modular.
	
	For the necessity, as a consequence of Proposition~\ref{prop:cond_nec_mod}, we just need to prove that if there exists a subalgebra, as stated, then $\mathcal{E}$ is not modular. Without loss of generality, suppose that $U=K+\spa\{e_i-e_{i+1},e_j+ e_{j+1}\}$, with $i<j$, is a subalgebra. Then $K+\spa\{e_i-e_{i+1},e_j- e_{j+1}\}$ is also closed under the product and, consequently a subalgebra. Nevertheless, as $\pi_{i(i+1)}(e_j)$, $\pi_{i(i+1)}(e_{j+1})\notin\langle e_i-e_{i+1}\rangle$, we have that 
	\begin{align*}
		\langle U,K+\spa\{e_i-e_{i+1},e_j- e_{j+1}\}\rangle&=K+\spa\{e_i,e_{i+1},e_j,e_{j+1}\}\\&\neq
		K+\spa\{e_i-e_{i+1},e_j,e_{j+1}\}\\&=U+\big(K+\spa\{e_i-e_{i+1},e_j- e_{j+1}\}\big).
	\end{align*}
	Since $U$ is not a quasi-ideal, $\mathcal{E}$ is not modular.
\end{proof}
\begin{remark}
	Notice that we cannot omit that $\pi_{i(i+1)}(e_j)$ or $\pi_{i(i+1)}(e_{j+1})\notin\langle e_i-e_{i+1}\rangle$. Let $\mathcal{E}$ be an evolution algebra with natural basis $\{e_1,e_2,e_3,e_4\}$ and product given by $e_1^2=-e_2^2=e_1+e_2$, $e_3^2=e_1-e_2+e_3+e_4$ and $e_4^2=e_1-e_2-e_3-e_4$. It is easy to check that $\mathcal{E}$ is solvable with maximum index of solvability and that its subalgebras are all quasi-ideals. Consequently, $\mathcal{E}$ is modular.
	
	\begin{minipage}{0.7\textwidth}\hspace{-0.4cm}
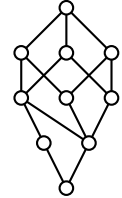
			\begin{tabular}{| l | l | l |}
				\hline
				\textbf{Subalg. of dim. $1$} & \textbf{Subalg. of dim. $2$} & \textbf{Subalg. of dim. $3$}\\
				\hline
				$\spa\{e_1+e_2\}$  & $\spa\{e_1,e_2\}$ & $\spa\{e_1,e_2,e_3+e_4\}$\\
				$\spa\{e_1-e_2\}$   & $\spa\{e_1-e_2,e_3+e_4\}$ & $\spa\{e_1,e_2,e_3-e_4\}$\\
				& $\spa\{e_1-e_2,e_3-e_4\}$ &$\spa\{e_1-e_2,e_3,e_4\}$\\
				\hline
			\end{tabular}
	\end{minipage}
	\begin{minipage}{0.25\textwidth}
		\begin{center}	\hspace{1.2cm}
			\begin{tikzpicture}[scale=0.6]
				\draw[thick] (0,-1.5) -- (-0.5,-0.5);
				\draw[thick] (-1,0.5) -- (-0.5,-0.5);
				\draw[thick] (0,-1.5) -- (0.5,-0.5);
				\draw[thick] (1,0.5) -- (0.5,-0.5);
				\draw[thick] (0,0.5) -- (0.5,-0.5);
				\draw[thick] (-1,0.5) -- (0.5,-0.5);
				\draw[thick] (0,1.5) -- (-1,0.5);
				\draw[thick] (1,0.5) -- (0,1.5);
				\draw[thick] (0,2.5) -- (0,1.5);
				\draw[thick] (-1,0.5) -- (-1,1.5);
				\draw[thick] (0,2.5) -- (-1,1.5);
				\draw[thick] (0,0.5) -- (-1,1.5);
				\draw[thick] (0,0.5) -- (1,1.5);
				\draw[thick] (0,2.5) -- (1,1.5);
				\draw[thick] (1,0.5) -- (1,1.5);
				\draw[fill=white,thick=black](0.5,-0.5) circle [radius=0.15];
				\draw[fill=white,thick=black](-0.5,-0.5) circle [radius=0.15];
				\draw[fill=white,thick=black] (0,-1.5) circle [radius=0.15];
				\draw[fill=white,thick=black] (1,0.5) circle [radius=0.15];
				\draw[fill=white,thick=black] (0,0.5) circle [radius=0.15];
				\draw[fill=white,thick=black] (-1,0.5) circle [radius=0.15];
				\draw[fill=white,thick=black] (1,1.5) circle [radius=0.15];
				\draw[fill=white,thick=black] (0,1.5) circle [radius=0.15];
				\draw[fill=white,thick=black] (-1,1.5) circle [radius=0.15];
				\draw[fill=white,thick=black] (0,2.5) circle [radius=0.15];
			\end{tikzpicture}
		\end{center}
	\end{minipage}
\end{remark}

\section*{Acknowledgments}
This work was partially supported by Agencia Estatal de Investigaci\'on (Spain),
grant PID2020-115155GB-I00 (European FEDER support included, UE) and
by Xunta de Galicia through the Competitive Reference Groups (GRC), ED431C
2023/31.

The third author was also supported by FPU21/05685 scholarship, Ministerio de Educaci\'on y Formaci\'on Profesional (Spain).

\end{document}